\documentclass{article}
\usepackage[english]{babel}
\usepackage[T1]{fontenc}
\usepackage[utf8]{inputenc}
\usepackage{amsthm}
\usepackage{enumerate}
\usepackage{amsfonts}
\usepackage{verbatim}
\usepackage{graphicx}
\usepackage[babel=true]{csquotes}
\usepackage{amsmath}
\usepackage{amssymb}
\usepackage{mathrsfs}
\usepackage{mathtools}
\newtheorem{defi}{Definition}
\newtheorem{theo}{Theorem}
\newtheorem{prop}{Proposition}

\newtheorem{lem}{Lemma}
\newtheorem{coro}{Corollary}
\newtheorem{remark}{Remark}
\title{Formal intercepts of Sturmian words}
\date{Ca\"{\i}us Wojcik}

\begin{document}

\maketitle

\begin{abstract}
We introduce the concept of formal intercept of Sturmian words, defined as an infinite sequence of integers written in Ostrowski expansion. We first recall the combinatorial proofs of basics properties of sturmian words. Then, we study the Rauzy graphs and repetition functions of Sturmian words. In the last part, we define the formal intercept associated to a sturmian word.
\end{abstract}

\bigskip

Sturmian words are defined as the infinite words having lowest unbounded complexity. They enjoy rich combinatorial structures, that are sometimes difficult to quantify. However, unlike most of other dynamical systems, in the case of sturmian words one can hope to find some explicit combinatorial formulas.

In order to describe the combinatorial properties of Sturmian words, we give here a combinatorial description of the second parameter in the caracterisation of sturmian words. The first parameter, well-known, is the slope, which is an irrational number in $]0,1[$, whose continued fraction expansion describes the set of factors of a sturmian word. The second parameter, the intercept, has been defined dynamicaly in the litterature, but its combinatorial implications in the structure of sturmian words were not well understood in the author's view.

In this paper, which is the first of two about formal intercepts, we define the formal intercept of a given Sturmian word, and show that there is a natural bijection between Sturmian words of a given slope, and the formal intercepts associated to this slope. Since we insist on the combinatorial properties of such parameters, and for sake of completeness, we start from the bottom about Sturmian words, and give entierely combinatorial proofs of their basic properties.

The paper is organized as follows. In the first section, we recall the proofs of basic properties of sturmian words. In the second part of the paper, we give a description of the factor graph and some links with the repetition function of Sturmian words. In the last part, we define the formal intercept of Sturmian word.

\section{Basic properties}

In this section we give combinatorial proofs of the main properties about sturmian words. All these come from the classical book \cite{lothaire}, rewritten in a concise manner.

\subsection{The Morse-Hedlund theorem}

Let $A$ be a finite set, the alphabet. A finite ord over $A$ is an element of the union $\cup_n A^n$, and if $u\in A^n$, we set $|u|=n$ and call it its length. For a letter $a\in A$ we note $|u|_a$ the number of occurrences of the letter $a$ in $u$.

An infinite word $x=x_1x_2x_3\ldots$ is an element of $A^{\mathbb{N}}$. A factor of $x$ is a finite word occurring in $x$. For such an infinite word, we note $\mathbb{P}_n(x)=x_1x_2\ldots x_n$ its prefix of length $n$, and we note $T(x)=x_2x_3x_4 \ldots $ the shifted of $x$, which consists of the infinite word $x$ deprived of its first letter. A suffix of $x$ is an element of the form $T^k(x)$ for some $k\geq 1$.

\begin{defi}
Let $x=x_1x_2x_3\ldots$ be an infinite word. For $n\geq 1$, set
\begin{center}
	$p(x,n)=Card\{x_ix_{i+1}\ldots x_{i+n-1}\ | \ i\geq 1\}$
\end{center}
and call $p(x,\cdot)$ the complexity function of $x$.
\end{defi}

\begin{theo}[Morse-Hendlund]
Let $x$ be an infinite word over $A$. Then $x$ is ultimately periodic if and only if there exists $n\geq 1$ such that $p(x,n)\leq n$.
\end{theo}

\begin{proof}
The condition is clearly necessary, since a ultimately periodic word has bounded complexity function. For the converse, from the increasing of the complexity function and the pigeonhole principle, there exists $n\in \mathbb{N}^*$ such that $p(x,n)=p(x,n+1)$. That means that every factor $u$ of $x$ can be only uniquely extend on the right. Let $u$ be a factor of $x$ having two occurrences in $x$. The two corresponding suffixes of $x$ are uniquely determined by $u$, and hence are equal, showing that $x$ is ultimately periodic.
\end{proof}

This theorem can be reformulated as follows : if $x$ is a non-ultimately periodic word, then $p(x,n)\geq n+1$ for all $n\geq 1$.

\begin{defi}
An infinite word is said to be Sturmian if
\begin{center}
	$\forall n\geq 1$, $p(x,n)=n+1$.
\end{center}
\end{defi}

Notice that a Sturmian word $x$ is a word over a $2$-letter alphabet, since $p(x,1)=2$, so we assume from now on that $A=\{0,1\}$. Also, notice that since $T(x)$ is non-ultimately periodic, $T(x)$ and $x$ share the same set of factors.

For all $n\geq 1$, there exists a unique factor $L_n$ of $x$ of length $n$ such that both $0L_n$ and $1L_n$ are factor of $x$, called the left special factor of $x$ of length $n$. Similarly, define the right special factor $R_n$
of $x$ as the unique factor of $x$ of length $n$ such that both $R_n0$ and $R_n1$ are factors of $x$.

\subsection{Slopes and balanced words}

\begin{defi}
A word over $A=\{0,1\}$ is said to be balanced if for all pair of factors $u,v$ with $|u|=|v|$ we have
\begin{center}
	$||u|_1-|v|_1|\leq 1$.
\end{center}
\end{defi}

\begin{lem}
A word $x$ is balanced if and only if for all palindrome $w$, $0w0$ and $1w1$ are not both factors of $x$.
\end{lem}

\begin{proof}
Note that the direct implication holds for any word $w$. For the converse, we take $u$ and $v$ of minimal length such that $||u|_1-|v|_1|> 1$, and we show that $\{u,v\}=\{0w0,1w1\}$ for some palindrome $w$.

The minimality hypothesis implies that $u$ and $v$ start with different letters, and let $w$ be the longest common prefix of $u^*$ and $v^*$, where the star denotes a word deprived of its first letter. Write
\begin{center}
	$u=awcu'$ \quad and \quad $v=bwdv'$
\end{center}
with $a,b,c,d\in\{0,1\}$ and $a\neq b$. Since the couple $(u',v')$ cannot satisfy the balanced property by minimality, we must have $c=a$ and $d=b$. From this the words $u'$ and $v'$ must be empty, $0w0$ and $1w1$ factors of $x$.

In order to show that $w$ is a palindrome, let $t$ be the longest common prefix of $w$ and $\tilde{w}$ where the tilda denotes the reversal of a word and assume by contradiction that $t\neq w$. Let $a$ be the letter following $t$ in $w$ and $\bar{a}$ the opposite letter. If $a=0$ then $0t0$ is a prefix of $0w0$ and $1\tilde{t}1$ is a suffix of $1w1$, contradicting the minimality hypothesis. If $a=1$ then writing $0w0=0t1u'$ and $1w1=v'0\tilde{t}1$ offers a pair $u',v'$ contradicting the minimality hypothesis.
\end{proof}

\begin{theo}
\begin{enumerate}[1)]
	\item An infinite balanced word satisfies $\forall n \in\mathbb{N}^*, \ p(x,n)\leq n+1$.
	\item An infinite word over $A=\{0,1\}$ is Sturmian if and only if he is balanced and non-ultimately periodic.
\end{enumerate}
\end{theo}

\begin{proof}
1) Let $x$ be a balanced word and suppose that there is $n\in\mathbb{N}^*$ such that $p(x,n+1) \geq p(x,n)+2$. Then $x$ has two distincts left special factors $y$ and $z$, and let $w$ be their longest common prefix. Then both $0w0$ and $1w1$ are factors of $x$, in contradiction with the balanced hypothesis.

2) A non-ultimately periodic word satisfy $\forall n \in\mathbb{N}^*, \ p(x,n)\geq n+1$ by the Morse \& Hedlund theorem. Combined with $1)$, we get the sufficient condition.

For the converse, we have to show that a Sturmian word $x$ is necessarily  balanced. Arguing by contradiction and using the Lemma, we assume that there is a palindrome $w$ such that both $0w0$ and $1w1$ are factors of $x$. The word $w$ is a right special factor of $x$, and one of the words $0w$ or $1w$ is a right special factor, and we assume without loss of generality that $0w$ is a right special factor. By the unicity of right special factors of a given length for Sturmian words, $1w$ is not a right special factor. So the three words $0w0$, $1w1$ and $0w1$ are factors of $x$ although $1w0$ is not.

Let $u=1w1v$ be a factor of $x$ with the prefix $1w1$ and $|v|=|w|$. We show that the right special factor of $x$ of length $|w|+1$, namely $0w$, is not a factor of $u$. Suppose this is the case and write the word $u$ as :
\begin{center}
	$u=1w1v=\lambda 0w \mu$
\end{center}
and let $t$ be the word such that $u=\lambda 0 t 1 v$. The word $t$ is both a prefix and a suffix of $w$, and we also see that $t1$ is a prefix of $w$ and that $0t$ is a suffix of $w$. But since $w$ is a palindrome, the $(|t|+1)^{\underline{th}}$ letter of $w$ is both a 0 and a 1, and that's a contradiction.

We've just shown that no factor of $u=1w1v$ of length $n=|w|+1$ is a right special factor. Since there is at most $n$ such factors, and $|u|=2n$, there is a factor $\nu$ of length $n$ in $u$ that occurs twice. Let's show that this implies that $x$ is ultimately periodic, hence proving the theorem. Let's note $y$ the suffix of $x$ beginning at the first occurrence of $\nu$, and $z$ the suffix of $x$ beginning at its second occurrence. The two infinite words $y$ and $z$ share the same prefix of length $n$, but since this prefix is not a right special factor, its following letter in $y$ and $z$ is uniquely determined, so that $y$ and $z$ share the same prefix of length $n+1$. And the letter following this prefix is also uniquely determined since the preceeding word of length $n$ appears in $u$ and so is not a right special factor. This argument goes on and on, so that we must have $y=z$. There are two suffixes of the infinite word $x$, taken at different starting point, that are equal. This shows that $x$ is ultimately periodic.

\end{proof}

The following proposition is an extension of the balanced property.

\begin{prop}
A word $x$ is balanced if and only if for all factors $u,v$ of $x$ we have
\begin{center}

	$\displaystyle \left| \frac{|u|_1}{|u|}-\frac{|v|_1}{|v|} \right|<\frac{1}{|u|}+\frac{1}{|v|}$.

\end{center}
\end{prop}

\begin{proof}
The sufficient condition is clear by taking two words of the same length. Conversely, we show the result by induction on $\max{|u|,|v|}$. If $|u|=|v|$ then the result follows directly from the balanced property. If $|u|>|v|$, write $u=st$ with $|s|=|v|$. From the balanced property on the one hand, and the induction hypothesis on the other, we have :
\begin{center}
	$\displaystyle \left| \frac{|s|_1}{|s|}-\frac{|v|_1}{|v|} \right|\leq\frac{1}{|v|}$
	
	\bigskip
	
	$\displaystyle \left| \frac{|t|_1}{|t|}-\frac{|v|_1}{|v|} \right|<\frac{1}{|t|}+\frac{1}{|v|}$
	
\end{center}
we also have by simple calculation :
\begin{center}
	$\displaystyle  \frac{|u|_1}{|u|}-\frac{|v|_1}{|v|} = \frac{|s|}{|u|}\left(\frac{|s|_1}{|s|}-\frac{|v|_1}{|v|}\right)+\frac{|t|}{|u|}\left(\frac{|t|_1}{|t|}-\frac{|v|_1}{|v|}\right)$
\end{center}
so that
\begin{center}
	$\displaystyle  \left|\frac{|u|_1}{|u|}-\frac{|v|_1}{|v|}\right| < \frac{|s|}{|u|}\times\frac{1}{|v|}+\frac{|t|}{|u|}\left(\frac{1}{|t|}+\frac{1}{|v|}\right)=\frac{1}{|u|}+\frac{1}{|v|}$
\end{center}
wich ends the proof.
\end{proof}

This proposition shows that the family of numbers $(|u|_1/|u|)$ behave like a Cauchy sequence when $u$ runs through the factors of a balanced word $x$. Hence we can define :

\begin{defi}
Let $x$ be a balanced word. We define the slople $\alpha$ of $x$ as the number
\begin{center}
	$\displaystyle\lim_{|u|\rightarrow +\infty}\frac{|u|_1}{|u|}$
\end{center}
where the limit is taken over the factors of $x$.
\end{defi}

By using this definition in the second caracterisation of balanced word, we get the speed relation :
\begin{center}
	$\displaystyle \left|\frac{|u|_1}{|u|}-\alpha\right|\leq \frac{1}{|u|}$
\end{center}
for every factor $u$ of a balanced word of slope $\alpha$.

If a word is ultimately periodic, then its slope is a rational number. The following theorem shows that the converse is true for balanced word.

\begin{theo}
\begin{enumerate}[1)]
	\item A balanced word is Sturmian if and only if its slope is irrational.
	\item Two balanced words of different slope only share a finite number of factors.
	\item Two Sturmian words of same slope have same set of factors.
\end{enumerate}
\end{theo}

\begin{proof}
1) Suppose that $\alpha=p/q$ is the slope of a Sturmian word $x$, with $p,q$ integers. Assume first that for every factor $u$ of $x$ such that $|u|=q$ we have $|u|_1=p$. If $w$ is any factor of $x$ of length $q+1$, then by assumption, the prefix and suffix of length $q$ have the same number of $1$'s, so that $w$ must begin and end with the same letter, showing that $x$ is ultimately periodic. Suppose now that there is an infinity of factors $u$ of $x$ such that $|u|=q$ and $|u|\neq p$. By the balanced property and without loss of generality we can assume $|u|_1=p+1$ for an infinity of such factors. Let $u$ and $v$ be two non-crossing such factors and $w=uzv$ a factor of $x$. From the relations :
\begin{center}
	$\displaystyle \left||w|_1-\frac{p}{q}|w|\right|=\left|2+|z|_1-\frac{p}{q}|z|\right|\leq 1$ \quad and \quad $\displaystyle\left||z|_1-\frac{p}{q}|z|\right|\leq 1$
\end{center}
we get
\begin{center}
	$\displaystyle |z|_1-\frac{p}{q}=-1$
\end{center}
then
\begin{center}
	$\displaystyle \left| \frac{|u|_1}{|u|}-\frac{|z|_1}{|z|} \right|=\left| \frac{|u|_1}{|u|}-\frac{p}{q}-\left(\frac{|z|_1}{|z|}-\frac{p}{q}\right) \right|=\left| \frac{1}{q}+\frac{1}{|z|} \right|=\frac{1}{|u|}+\frac{1}{|z|}$
\end{center}
contradicting the strict inequality in the second caracterisation of balanced words.

2) The speed relation implies that for two distincts slopes $\alpha$ and $\beta$, a finite word that is too long cannot be a common factor of two balanced words of slope $\alpha$ and $\beta$.

3) We first show that two Sturmian words have same set of left special factors. Let $x$ and $y$ be two Sturmian words of slope $\alpha$. Write $L_n(x)$ and $L_n(y)$ for the left special factor of length $n\geq 1$ of $x$ and $y$ respectively. From the strict inequality
	\begin{center}
	$\displaystyle \left|\frac{|u|_1}{|u|}-\alpha\right|< \frac{1}{|u|}$
\end{center}
satisfied by the $2$-letters factors
\begin{center}
	$0L_1(x)$, $1L_1(x)$, $0L_1(y)$, $0L_1(y)$, 
\end{center}
we see that $2\alpha$ must lie in the two open balls of radius $1$ and center $|L_1(x)|_1$ and $|L_1(x)|_1+1$, determining uniquely the number $|L_1(x)|_1$ which has to be equal to $|L_1(y)|_1$. Since they are both letters, we have $|L_1(x)|_1=|L_1(y)|_1$. Recall that two left special factors of a Sturmian word are prefix of one another, so that we can prove the result by induction : suppose that $L_{n-1}(x)=L_{n-1}(y)$, and by the same argument we have $|L_n(x)|_1=|L_n(y)|_1$ and so $L_n(x)=L_n(y)$. Let $c_\alpha=\lim L_n(x)=\lim L_n(y)$, then $c_\alpha$ is a balanced word that is not ultimately periodic since its slope is irrationnal, so $c_\alpha$ is a Sturmian word, and by the cardinality of the sets of factors involved, we see that $x$, $y$ and $c_\alpha$ share the same set of factors.
\end{proof}

\begin{defi}
For all Sturmian word of slope $\alpha$, the sequence $(L_n)$ of its left special factors defines a Sturmian word :
\begin{center}
	$c_\alpha=\lim L_n$
\end{center}
which depends only on the slope $\alpha$, noted $c_\alpha$ and called the caracteristic Sturmian word of slope $\alpha$.
\end{defi}

\begin{prop}
Let $(L_n)$ and $(R_n)$ be the sequences of left special factors and right special factors respectively of a Sturmian word. Then :
\begin{enumerate}[1)]
  \item The Sturmian word $x$ is caracteristic if and only if both $0x$ and $1x$ are Sturmian.
	\item $c_\alpha=\lim \widetilde{R_n}$
	\item $\forall n\in\mathbb{N}^*$, $R_n=\widetilde{L_n}$
	\item The set of factors of a Sturmian word is stable under reversal,
	\item For all Sturmian word $x$, at least one of the words $0x$ and $1x$ is Sturmian
\end{enumerate}
\end{prop}

\begin{proof}
1) If both $0x$ and $1x$ are Sturmian, then all the prefixes of $x$ are left special, so $x=\lim L_n=c_\alpha$.

2) Since the right special factors are suffixes of one another, the word $c=\lim\widetilde{R_n}$ is well-defined, balanced and of irrational slope, so it is Sturmian. Besides, both $0c$ and $1c$ are Sturmian, so that $c=c_\alpha$ by 1).

3) Obvious since $c_\alpha=\lim L_n=\lim \widetilde{R_n}$.

4) Obvious from $3)$ and the fact that a Sturmian word and the caracteristic word of same slope share the same set of factors.

5) It is clear if $x$ is caracteristic. Let $u$ be a prefix of $x$ that is not left special. Then by $4)$ there is a unique letter $a\in\{0,1\}$ such that $au$ is a factor of $x$, and this letter does not depend on $u$. The word $ax$ is then balanced and non-ultimately periodic, so it is Sturmian.
\end{proof}

\subsection{Caracteristic words and continued fractions}

Recall that every irrational number $\alpha\in ]0,1[$ can be written uniquely in the form
\begin{center}
	$\alpha =[0;a_1,a_2,\ldots]=\displaystyle \cfrac{1}{a_1+\cfrac{1}{a_2+\cfrac{1}{a_3+\ldots}}}$
\end{center}
with $a_i\in\mathbb{N}^*$ for $i\geq 1$. The coefficient $(a_i)$ are called the partial quotient of $\alpha$.

We define the positive integers $p_n$ and $q_n$ as the irreducible quotient
\begin{center}
	$\displaystyle \frac{p_n}{q_n}=[0;a_1,\ldots,a_n]$
\end{center}
and we set $q_{-1}=0$ and $q_0=1$. We call the sequence $(q_n)$ the sequence of continuant of $\alpha$. We have the induction relation
\begin{center}
	$q_{n+1}=a_{n+1}q_n+q_{n-1}$
\end{center}
for $n\geq 0$. Notice that for all $n\geq 0$, $q_{n+1}$ and $q_n$ are relatively prime (the induction steps are the steps of Euclide's algorithm).

\begin{theo}
Let $\alpha=[0;a_1,a_2,\ldots]$ be an irrational number in $]0,1[$. Define the sequence of words :
\begin{center}
	$s_{-1}=1$, \quad $s_0=0$, \quad $s_1=s_0^{a_1-1}s_{-1}$,
	
	\bigskip
	
	$s_{n+1}=s_n^{a_{n+1}}s_{n-1}$
\end{center}
for all $n\geq 1$. Then :
\begin{center}
	$c_{\alpha}=\lim s_n$.
\end{center}
\end{theo}

\begin{proof}
Define the two morphisms
\begin{center}
	$E$ : $\begin{matrix}
		 0 & \longmapsto & 1  \\
		 1 & \longmapsto & 0  \\
		\end{matrix}$ \quad and \quad $G$ : $\begin{matrix}
		 0 & \longmapsto & 0  \\
		 1 & \longmapsto & 01  \\
		\end{matrix}$,
\end{center}
they are injective, in the sense that if $x$ and $y$ are two infinite words such that $G(x)=G(y)$, then $x=y$, and the same for $E$. We obviously have that $x$ is sturmian if and only if $E(x)$ is sturmian. Let's show now that $x$ is Sturmian if and only if $G(x)$ is Sturmian.

 Suppose that $G(x)$ is unbalanced : there exists a palindrome $w$ such that both $0w0$ and $1w1$ are factors of $G(x)$. In view of $G$ there must exist a word $z$ such that $w=0z0$, moreover $01w1=010z01$ is a factor of $G(x)$. There must be a word $y$ such that $0z=G(y)$ and $01w1=G(1y1)$ so that $1y1$ is a factor of $x$ by injectivity of $G$. On the other hand $0w0=00z00=G(0y0)$ is a factor of $G(x)$ and so $1y1$ is a factor of $x$. Both $0y0$ and $1y1$ are factors of $x$ so $x$ is unbalanced. This shows that if $x$ is Sturmian then $G(x)$ is balanced, and it is not hard to see that its slope is irrational, so $G(x)$ is Sturmian.

Conversely, if $G(x)$ is Sturmian, then $x$ is Sturmian. Indeed, suppose that $x$ is unbalanced, namely let $w$ be a palindrome such that both $0w0$ and $1w1$ are factors of $x$. Then both $0G(w)0$ and $01G(w)01$ are factors of $G(x)$. In view of $G$, $0G(w)00$ is a prefix of $G(0w0a)$ for any letter $a$, so that both $0G(w)00$ and $1G(w)01$ are factors of $G(x)$, showing at once that $G(x)$ is unbalanced. It is clear from the slopes that if $G(x)$ is not ultimately periodic, then $x$ is also not ultimately periodic.

Let $m$ be the greatest $m\geq 1$ such that $0^m1$ is a factor of $c_\alpha$. Suppose $m\geq 2$, by the balanced property, the words $10^{k}1$ for $k=0\ldots m-2$ cannot be factors of $c_\alpha$, and we see that $10^{m-1}1$ must be a factor of $c_\alpha$ for otherwise $c_\alpha$ would be ultimately periodic. If $m=1$, then we easily see that $11$ must be a factor of $c_\alpha$. So the word $0^{m-1}1$ is left special and hence a prefix of $c_\alpha$. All this sums up to the fact that $c_\alpha$ can be factorised in an infinite concatenation of the two words $0^{m-1}1$ and $0^{m-1}10$ for some $m\geq 1$.

We define the morphisms, for $m,n\geq 1$ :
\begin{center}
	$\theta_m=G^{m-1}\circ E \circ G$ \quad and \quad $h_n=\theta_{a_1}\circ \theta_{a_2}\circ\cdots\circ\theta_{a_n}$.
\end{center}
Since $\theta_m(0)=0^{m-1}1$ and $\theta_m(1)=0^{m-1}10$, we have seen that $c_\alpha$ factorises as $c_\alpha=\theta_m(x)$ for some $x$, that must be Sturmian.
For $m\geq 1$, we have $\theta_m(0c_\alpha)=0^{m-1}1\theta_m(c_\alpha)$ and $\theta_m(1c_\alpha)=0^{m-1}10\theta_m(c_\alpha)$ so that $\theta_m(c_\alpha)$ is caracteristic and, according to the slopes, we have
\begin{center}
	$\theta_m(c_\alpha)=\displaystyle c_{\frac{1}{m+\alpha}}$.
\end{center}
so that for all $n\geq 1$ we have :
\begin{center}
	$\displaystyle h_n(c_{[0;a_{n+1},a_{n+2},\ldots]})=c_{\alpha}$.
\end{center}
Moreover, we have $h_n(0)=s_n$ and $h_n(1)=s_ns_{n-1}$ as it is easily checked by induction on $n\geq 1$. This shows that $s_n$ is a prefix of $c_{\alpha}$ for all $n\geq 1$, proving the theorem.
\end{proof}

\subsection{Standard and central words}

\begin{defi}
The subset of $(A^*)^2$ of standard pairs is recursively defined by the rules :
\begin{itemize}
	\item $(0,1)$ is a standard pair,
	\item if $(u,v)$ is a standard pair, then $(vu,v)$ and $(u,uv)$ are standard pairs.
\end{itemize}
\end{defi}

We recall the notation $x^{-}$ for a word $x$ deprived of its last letter. If $x$ is empty, then we set $x^{-}$ to be the empty word.

\begin{prop}
Let $(u,v)$ be a standard pair. 
\begin{enumerate}[1)]
  \item $(uv)^{--}=(vu)^{--}$,
	\item if $|u|\geq 2$, $u$ ends with $10$. If $|v|\geq 2$, $u$ ends with $01$
	\item $u^{--}$ and $v^{--}$ are palindromes.
	\item We have $|u||v|_1-|u|_1|v|=1$.
\end{enumerate}
\end{prop}

The proofs of proposition $3$ are straightforward inductions.

\begin{defi}
A word is said to be standard if it is a coponent of a central pair.
\end{defi}

\begin{prop}\begin{enumerate}
	\item If $u$ is standard, then $u^{--}$ is palindromic.
	\item A standard word is primitive (that is, not a non-trivial power of a word).
	\item The words $(s_n)$ in theorem 4 are standard. The suffix of length $2$ of $s_n$ is $t_n$, where $t_n=10$ if $n$ is even, and $t_n=01$ if $n$ is odd, for $n\geq 2$.
\end{enumerate}
\end{prop}

\begin{proof}
1) The fact that $u^{--}$ is palindromic is trivial from proposition $3$.

2)The word $u$ is primitive since by proposition $3-4)$, $|u|$ and $|u|_1$ are coprime.

3) We see by the definition of the sequence $(s_n)$ that $(s_{2n},s_{2n-1})$ and $(s_{2n},s_{2n+1})$ are standard pairs for all $n\geq 0$. The remaining part of the assertion is clear by proposition $3$.
\end{proof}

\begin{defi}
We define the set of central words by one of the following equivalent definitions :
\begin{enumerate}[(i)]
	\item a word $w$ is central if and only if there exists a standard word $u$ such that $w=u^{--}$,
	\item the set of central words is inductively defined as follows :
	\begin{itemize}
		\item powers of a letter are central words
		\item if $p$ and $q$ are central, and $p01q$ is a palindrome, then $p01q$ is central
		\end{itemize}
		\item a word $w$ is central if and only if it is a power of a letter, or a palindrome of the form $p01q$ with $p$, $q$ palindromes,
		\item a word is central if and only if it is a prefix palindrome of a caracteristic word.
\end{enumerate}
The decomposition $w=p01q$ with $p$, $q$ palindrome of a central word that is not a power of a letter is then unique.
\end{defi}

\begin{proof}
\underline{$(ii)$-central $\Leftrightarrow$ $(iii)$-central :} It is clear from the definition $(ii)$ that $(ii)$-central words are palindromes, so that $(ii)$-central $\Rightarrow$ $(iii)$-central. For the converse, it is sufficient to show that if $w=p01q$ is a palindrome with $p$ and $q$ palindrome is $(iii)$-central, then $p$ and $q$ are $(iii)$-central. We cannot have $|p|=|q|$ since $w$ is a palindrome, and we can assume that $|p|\leq|q|-1$. If $|p|=|q|-1$ then $q=p0=0p$ and $p$, $q$ are powers of letters. If $|p|=|q|-2$ then $q=p01u$ and since $w=p01\tilde{u}10p$ is a palindrome, $u$ is a palindrome and $q$ is $(iii)$-central. By continuing this argument with $q$ in the place of $w$, we see that there exists a unique $N\geq 1$ such that $q=(p01)^Nt$ with $|t|\leq |p|-1$, so that $p01t$ is $(iii)$-central with $|t|\leq |p|-1$, and in this situation we have seen that $p$ is $(iii)$-central.

\underline{$(ii)$-central $\Rightarrow$ $(i)$-central :} The case of powers of letters being obvious, we show by induction on $|w|$ that if $w=p01q$ with $p$, $q$ and $w$ $(ii)$-central, then $(q10,p01)$ is a standard pair. We can assume $|p|\leq |q|$ without loss of generality. If $|p|=|q|-1$, then $p0=q=0p=0^{|q|}$ and $(0^{|q|}10,0^{|q|}1)$ is a standard pair. If $|p|\leq|q|-2$, then $q=p01u$ for some palindrome $u$. Since $q$ is $(ii)$-central, it is $(iii)$-central and from the preceeding proof we know that $u$ is $(iii)$-central, and so $u$ is $(ii)$-central. By the induction hypothesis, $(u10,p01)$ is a standard pair, and so is $(p01u10,p01)=(q10,p01)$. Since $(q10,p01)$ is a standard pair, $(p01q10,p01)$ is also and $w=(p01q10)^{--}$ is $(i)$-central.

\underline{$(i)$-central $\Rightarrow$ $(iii)$-central :} Let $w=u^{--}$ with $(u,v)$ a standard pair. Write $u=w01=yx$ for a standard pair $(x,y)$, $x=q10$ and $y=p01$, with $p$ and $q$ palindromes. Then $w=p01q$ and $w$ is $(iii)$-central.

\underline{$(i)$-central $\Rightarrow$ $(iv)$-central :} We know that a $(i)$-central word is a palindrome. Let $\Gamma : (u,v)\in (A^*)^2 \mapsto (u,uv)\in (A^*)^2$ and $\Delta : (u,v)\in (A^*)^2 \mapsto (vu,v)\in (A^*)^2$. Let $w=u^{--}$ with $(u,v)=\Gamma^{a_k}\circ\Delta^{a_{k-1}}\circ\cdots\circ\Delta^{a_2}\circ\Gamma^{a_{1}-1}(0,1)$, with $a_i\in\mathbb{N}^*$ for $i=1\ldots k$. Then $w=s_k^{--}$ for any caracteristic word having slope whose partial quotients begins with $a_1,a_2,\ldots,a_k$, with the sequence $(s_n)$ defined as in theorem 4, so that $w$ is $(iv)$-central. The case where $w=v^{--}$ with $(u,v)$ standard is similar.

\underline{$(iv)$-central $\Rightarrow$ $(iii)$-central :}
In view of the preceeding proof, any $(iv)$-central word is a prefix of a $(i)$-central word, and so a prefix of a $(iii)$-central word. Let $w$ be a palindrome prefix of a palindrome $p01q$ with $p$ and $q$ palindromes. We may assume $|p|\leq |q|$ and by induction on $|p01q|$ we may assume $|w|>|q|$ since otherwise $w$ is a prefix of the $(iii)$-central word $q$. If $|w|=|q|+1$ then $w=q1=1q$ and $w$ is a power of a letter. If $|w|\geq|q|+2$ then write $w=q10t$, and since $w$ is a prefix of $q10p$, $t$ is a prefix of $p$, and so is also a prefix of $q$. The word $t$ is a suffix and a prefix of the palindromic word $w$, and so is palindromic, and $w=q10t$ with $q$ and $t$ palindromes, as required.

\underline{Unicity of decomposition :} Let $w=p01q=s01t=u^{--}$ be a central word, with $p$, $q$, $s$ and $t$ palindrome and $u10$ standard, and assume that $|s|>|p|$. We cannot have $|s|=|p|+1$, so write $s=p01\lambda$ and see that $q=\lambda 01t$ so that $u=w10=q10p10=\lambda 01t10p10=t10s10=t10p01\lambda10$, and the two words $\lambda 01$ and $t10p10$ commute. The primitive word $u$ is a product of two non-empty commuting words, hence a contradiction.

\end{proof}

\section{Repetition function and Rauzy graphs of Sturmian words}

We recall the following notations. The dynamical map $T$ is the shift, which removes the first letter of an infinite word. For any word $x$ and integer $n\geq 1$, we note $\mathbb{P}_n(x)$ the prefix of length $n$ of $x$.

\subsection{Definitions}

In \cite{bugeaud} a new complexity function is introduced, also called the repetition function. We define here a similar function and still call it the repetition function, since the two are linked by a simple formula. Namely, if $r_0(x,n)$ is Bugeaud and Kim's repetition function, then we have $r_0(x,n)=n+r(x,n)$.

\begin{defi}[Repetition function]
Let $x$ be an infinite word over a finite alphabet $A$. We define, for an integer $m>0$ :
\begin{center}
	$r(x,m)=\max\{k\in\mathbb{N} \ | \ \mathbb{P}_m(x), \mathbb{P}_m(T(x)), \ldots, \mathbb{P}_m(T^{k-1}(x)) \text{ are all distincts }\}$.
\end{center}
The function $r(x,\cdot)$ is called the repetition function of $x$.
\end{defi}

\begin{prop}
Let $x$ be an infinite word.
\begin{itemize}
	\item $\forall m >0$, $r(x,m)\leq p(x,m)$.
	\item if $x$ is Sturmian we have $\forall m >0$, $r(x,m)\leq m+1$
\end{itemize}
\end{prop}

\begin{defi}[Rauzy graph]
Let $x$ be an infinite word over an alphabet $A$. For every integer $m>0$, we define the factor graph, or Rauzy graph, of $x$ of degree $m$ as the directed graph having :
\begin{itemize}
	\item vertexes as the factors of $x$ of length $m$
	\item an arrow $s\rightarrow t$ if and only if there exists a factor $r$ of $x$ of length $m+1$ such that $s$ is a prefix of $r$ and $t$ a suffix of $r$.
\end{itemize}
\end{defi}

Given a path $s_1\rightarrow s_2 \rightarrow \ldots \rightarrow s_k$ in this graph, we set $k-1$ to be its length. The path defined by $x$ in $G_m$ is the infinite path
\begin{center}
	$\mathbb{P}_m(x)\rightarrow \mathbb{P}_m(T(x))\rightarrow \ldots \rightarrow \mathbb{P}_m(T^k(x))\rightarrow \ldots$.
\end{center}

For a Sturmian word $x$ and $m>0$, $G_m$ has $m+1$ vertexes. The vertex $L_m$ has in-degree $2$ and the vertex $R_m$ has out-degree $2$ (notice that they may be equal). Every vertex that is neither $L_m$ nor $R_m$ has in-degree $1$ and out-degree $1$. Therefore, $G_m$ is the fusion of two cycles, sharing a common path. The following proposition explains how to read the repetition function on the factor graph of $x$. A Hamiltonian path in a directed graph is a path that does not visit a vertex more that twice.
\begin{prop}
Let $x$ be a Sturmian word, $m>0$ and $G_m$ its Rauzy graph of degree $m$. Then $r(x,m)$ is the length of the longest Hamiltonian finite path
\begin{center}
	$\mathbb{P}_m(x)\rightarrow \mathbb{P}_m(T(x))\rightarrow \ldots \rightarrow \mathbb{P}_m(T^{k-1}(x))$.
\end{center}
in the infinite path defined by $x$.
\end{prop}

\subsection{Repetition function of caracteristic words}

\begin{theo}
Let $x$ be a Sturmian word and $m\geq 2$. The following statements are equivalents :
\begin{enumerate}[i)]
	\item $r(x,m)=m+1$
	\item $r(x,m)\neq r(x,m-1)$
\end{enumerate}
\end{theo}

\begin{proof}
The implication $(i)\Rightarrow(i)$ is clear since $r(x,m-1)\leq m$. For the converse, let $A_m$ and $B_m$ be the two distinct vertexes of $G_m$ such that
\begin{center}
	$R_m\rightarrow A_m$ \quad and \quad $R_m\rightarrow B_m$.
\end{center}

Consider the path
\begin{center}
	$\mathbb{P}_{m-1}(x)\rightarrow \mathbb{P}_{m-1}(T(x))\rightarrow \ldots \rightarrow \mathbb{P}_{m-1}(T^{r(x,m-1)}(x))$.
\end{center}
in $G_{m-1}$. There exists a unique integer $0\leq j < r(x,m-1)$ such that $\mathbb{P}_{m-1}(T^{r(x,m-1)}(x))=\mathbb{P}_{m-1}(T^j(x))$. In $G_m$, we cannot have $\mathbb{P}_{m}(T^{r(x,m-1)}(x))=\mathbb{P}_{m}(T^j(x))$ because this would imply $r(x,m)= r(x,m-1)$, which by assumption is not the case. We then have $\mathbb{P}_{m}(T^{r(x,m-1)}(x))\neq\mathbb{P}_{m}(T^j(x))$ and these two words differ only by their last letters. This shows that
\begin{center}
	$\{A_m,B_m\}=\{\mathbb{P}_{m}(T^{r(x,m-1)}(x)),\mathbb{P}_{m}(T^j(x))\}$
\end{center}
so that the path
\begin{center}
	$\mathbb{P}_{m}(x)\rightarrow \mathbb{P}_{m}(T(x))\rightarrow \ldots \rightarrow \mathbb{P}_{m}(T^{r(x,m)-1}(x))$.
\end{center}
passes on the two vertexes $A_m$ and $B_m$. This path is the longest Hamiltonian path that starts at $\mathbb{P}_{m}(x)$ in the path defined by $x$, so we can see that it must pass by all the $m+1$ vertexes of $G_m$. This shows that $r(x,m)=m+1$.
\end{proof}

\begin{lem}
Let $c_\alpha$ be a caracteristic Sturmian word. Then we have
\begin{center}
	$\mathbb{P}_{m}(T^{r(c_\alpha,m)}(c_\alpha))=\mathbb{P}_m(c_\alpha)=L_m$
\end{center}
for all $m>0$.
\end{lem}

\begin{proof}
The second equality comes from the definition of $c_\alpha$. Let $0\leq j < r(c_\alpha,m)$ be the only integer such that $\mathbb{P}_{m}(T^{r(c_\alpha,m)}(c_\alpha))=\mathbb{P}_{m}(T^j(c_\alpha))$ and assume $j\neq 0$. Then $\mathbb{P}_{m}(T^{r(c_\alpha,m)-1}(x))\neq\mathbb{P}_{m}(T^{j-1}(c_\alpha))$ and these two words differ only by their first letters. This shows that $\mathbb{P}_{m}(T^j(c_\alpha))$ is left special, so that $j=0$ and this is a contradiction.
\end{proof}

We define $r(z,m)$ for a finite word $z$ and $m>0$, provided $z$ admits a factor of length $m$ that occurs at least twice, as $r(x,m)$ for any infinite word $x$ such that $z$ is a prefix of $x$.

\begin{lem}
Let $z=p01q$ be a central word with $|p|\leq |q|$. Then
\begin{center}
	$r(z,|p|+1)=|p|+2$.
\end{center}
\end{lem}

\begin{proof}
Let $c_\alpha$ be a caracteristic word having $z$ as a prefix. Then by the preceeding lemma we have $\mathbb{P}_{|p|+1}(T^{r(z,|p|+1)}(c_\alpha))=p0$

We prove the result by induction on $|z|$. Since $|z|$ is palindromic we cannot have $|p|=|q|$ and if $|p|=|q|-1$ then $q=p0=0p$ so $z=0^{|p|+1}10^{|p|+1}$ and the result is clear. Assume that $|p|\leq|q|-2$ and write $q=p01u$, $u$ is palindromic since $z=q10p=p01u10p$ is palindromic so that $q=p01u$ is the decomposition of $q$ as a central word.

If $|p|\leq |u|$ then we are done by induction. Assume that $|u|\leq |p|$ so that $r(z,|u|+1)=r(q,|u|+1)=|u|+2$ by induction. Since $r(z,|u|+1)\leq r(z,|p|) \leq |u|+2$ we have $r(z,|p|)=|u|+2$. But $z=u10p10p$ so that $u10p0$ is not a prefix of $z$ and we must have $r(z,|p|+1)>r(z,|u|+1)=|u|+2=r(z,|p|)$, and hence $(z,|p|+1)\neq r(z,|p|)$. By theorem 5, we have $r(z,|p|+1)=|p|+2$.
\end{proof}

\begin{coro}
Let $c_\alpha$ be the caracteristic Sturmian word of slope $\alpha$, and let $(q_n)$ be the sequence of continuant of $\alpha$. Then for all $n\geq 0$ we have
\begin{center}
	$r(c_\alpha)=q_n$ \quad for all \quad $q_n-1\leq m \leq q_{n+1}-2$. \quad \quad $(m\neq 0)$
\end{center}
\end{coro}

\begin{proof}
Let $(s_n)$ be the sequence associated to $\alpha$ defined as in theorem $4$, so that $c_\alpha=\lim s_n$. It is easily checked that $|s_n|=q_n$ for $n\geq 0$. We have
\begin{center}
	$c_\alpha=\lim s_{n+2}= \lim s_{n+1}s_n=\lim s_{n+1}s_n^{--}=\lim s_n^{--}t_ns_{n+1}^{--}$
\end{center}
where $t_n=10$ if $n\geq 2$ is even and $t_n=01$ if $n\geq 2$ is odd. The words $s_n^{--}t_ns_{n+1}^{--}$ are the central prefixes of $c_\alpha$ and we have $r(c_\alpha,|s_n^{--}|+1)=|s_n^{--}|+2=|s_n|$ and since $r(c_\alpha,|s_{n+1}^{--}|)\leq |s_n|$, we have
\begin{center}
	$r(c_\alpha,m)=|s_n|=q_n$
\end{center}
for $n\geq 2$ and $q_n-1\leq m \leq q_{n+1}-2$.

If $a_1\geq 3$, then it is easily checked that the formula still holds for $1\leq m \leq q_2-2$. If $a_1=2$, or $a_1=1$ and $a_2\geq 2$, then the formulas hold but the set of integer $m$ such that $1\leq m \leq q_1-2$ is empty. If $a_1=1$ and $a_2=2$ then the formulas hold, but the sets of integers $m$ such that $1\leq m \leq q_1-2$ or $q_1-1\leq m \leq q_2-2$ are empty.
\end{proof}

\subsection{Rauzy graph of Sturmian words}

Let $x$ be a Sturmian word of slope $\alpha$ whose sequence of continuant is $(q_n)$.

\bigskip

Notations : \begin{itemize}
	\item In the remaining part of the article, we make the abuse of notation of noting $[a,b]$ the integer interval of integers $m$ such that $a\leq m \leq b$.
	\item We define the integer intervals $I_n$, for $n\geq 0$,
	\begin{center}
		$I_n=[q_n-1,q_{n+1}-2]$
		
    $I_n^0=[q_n-1,q_n+q_{n-1}-2]$
	\end{center}
and for $1\leq l \leq a_{n+1}-1$,
\begin{center}
	$I_n^l=[lq_n+q_{n-1}-1,(l+1)q_n+q_{n-1}-2]$.
\end{center}
If $a_1=1$ or $a_1=2$ then $I_0$ is empty. If $a_1=1$ and $a_2=1$, then both $I_0$ and $I_1$ are empty.
\item An Eulerian path in a directed graph is a path that does not pass twice on the same arrow. A cycle in a directed graph is an Eulerian path $s_1\rightarrow s_2 \rightarrow \ldots \rightarrow s_k$ such that $s_1=s_k$ and we set $k$ to be its length.
\end{itemize}

We recall the notation $u^*$ for a finite word $u$, denoting the suffix of length $|u|-1$ of $u$, which is $u$ deprived of its first letter.

\begin{prop}
Let $m\in I_n^l$ for $n\geq 0$ and $1\leq l \leq a_{n+1}-1$, then :
\begin{enumerate}
	\item one of the two cycles of $G_m$ is of length $q_n$. It is called the referent cycle.
	\item the other cycle is of length $lq_n+q_{n-1}$.
	\item The arrow $R_m\rightarrow R_m^*t_{n-1}^{-}$ belongs to the referent cycle, and the arrow $R_m\rightarrow R_m^*t_{n}^{-}$ belongs to the non-referent cycle. These two arrows do not belong to the same cycle.
\end{enumerate}
\end{prop}

\begin{proof}
1) Since two infinite words sharing the same set of factors also share the same Rauzy graphs, we can reduce to the case $x=c_\alpha$. Since $r(c_\alpha,m)=q_n$ by Corollary $2$, by definition of the repetition function the path
\begin{center}
	$\mathbb{P}_{m}(c_\alpha)\rightarrow \mathbb{P}_{m}(T(c_\alpha))\rightarrow \ldots \rightarrow \mathbb{P}_{m}(T^{r(c_\alpha,m)}(c_\alpha))$
\end{center}
defines a cycle of length $q_n$.

2) The common part of the two cycles is the shortest path that starts at the vertex $L_m$ and ends at the vertex $R_m$. The finite word $w$ defined by this path is left and right special, so it is the shortest central word of length $|w|\geq m$, and this length is equal to $(l+1)q_n+q_{n-1}-2$ and has $(l+1)q_n+q_{n-1}-1$ vertexes. The path so defined is of length $(l+1)q_n+q_{n-1}-2-m$. Since the sum of the length of the two cycles equals the sum of the number of vertexes of $G_m$ and the number of vertexes in the common part, we get that the other cycle is of length $\mu$ where
\begin{center}
	$q_n + \mu = m+1 + (l+1)q_n+q_{n-1}-1-m $
\end{center}
so that $\mu=lq_n+q_{n-1}$.

Since $q_n$ and $lq_n+q_{n-1}$ are coprime, the referent cycle is well-determined by its length.

3) The common path $L_m\rightarrow \ldots \rightarrow R_m$ corresponds to the central word of length $(l+1)q_{n}+q_{n-1}-2$, namely $s_n^{l+1}s_{n-1}^{--}$. The referent cycle is the cycle
\begin{center}
	$\mathbb{P}_{m}(c_\alpha)\rightarrow \mathbb{P}_{m}(T(c_\alpha))\rightarrow \ldots \rightarrow \mathbb{P}_{m}(T^{r(c_\alpha,m)}(c_\alpha))$
\end{center}
so we only have to see that $s_n^{l+1}s_{n-1}^{-}$ is a prefix of $c_\alpha$ since $s_{n-1}$ ends with $t_{n-1}$. But this is obvious, $s_{n-1}$ is a prefix of $s_n$, and $s_{n+1}=s_n^{a_{n+1}}s_{n-1}$, so that indeed the arrow $R_m\rightarrow R_m^*t_{n-1}^{-}$ belongs to the referent cycle. The fact that $R_m\rightarrow R_m^*t_{n}^{-}$ belongs to the non-referent cycle comes from the fact that $s_n^{l+1}s_{n-1}^{--}t_n$ is not a prefix of $c_\alpha$. It is obvious that the two arrows leaving the right special factor $R_m$ cannot be on the same cycle.

\end{proof}

\begin{defi} \begin{itemize}
	\item We say $x$ turns around a cycle of length $k$ in $G_m$ when $r(x,m)=k$ and the path $\mathbb{P}_m(x)\rightarrow \mathbb{P}_m(T(x))\rightarrow \ldots \rightarrow \mathbb{P}_m(T^k(x))$ shares the same arrow as this cycle.
	\item We say $x$ turns $d$ times around a cycle of length $k$ if for all $i=0\ldots d-1$, $T^{ik}(x)$ turns around this cycle.
\end{itemize}
\end{defi}

For a Sturmian word $x$, since the two cycles of its Rauzy graph $G_m$ have different length, $x$ turns around a cycle of length $k$ if and only if $r(x,m)=k$.

\begin{theo}
For $m\in I_n^l$, the caracteristic word $c_\alpha$ turns around the referent cycle $a_{n+1}-l$ times, and no more.
\end{theo}

\begin{proof}
We first consider the case where $l=0$. Then the central word $s_n^{--}$ is a strict prefix of $L_m$ and $L_m$ is a strict prefix of the central word $s_n^{--}t_ns_{n-1}^{--}$. The word $z=s_{n+1}s_n^{--}=s_n^{a_{n+1}+1}s_{n-1}^{--}$ is central and so we have
\begin{center}
	$r(z,m)=q_n=r(T^{q_n}(z),m)=\ldots=r(T^{(a_{n+1}-1)q_n}(c_\alpha),m)$,
\end{center}
showing that $c_\alpha$ turns at least $a_{n+1}$ times around the referent cycle.

Since $s_{n+1}s_n$ is a prefix of $c_\alpha$, $s_{n+1}s_n=zt_n=s_n^{a_{n+1}+1}s_{n-1}^{--}t_n$ is a prefix of $c_\alpha$ and $s_ns_{n-1}^{--}t_n$ is a prefix of $T^{a_{n+1}q_n}(c_\alpha)$ and from this we easily see that the word $T^{a_{n+1}q_n}(c_\alpha)$ passes by the arrow $R_m\rightarrow R_m^*t_{n}^{-}$ before passing by the arrow $R_m\rightarrow R_m^*t_{n-1}^{-}$. This shows that $c_\alpha$ does not turn $a_{n+1}+1$ times around the referent cycle.

The case $l>0$ is similar.
\end{proof}

\begin{lem}
Let $x$ be a Sturmian word of slope $\alpha$, and let $m>0$. Then $x$ does not turn twice around the non-referent cycle.
\end{lem}

\begin{proof}
Since the set of factors of a Sturmian word is stable under reversal, we see that if $s\rightarrow t$ is an arrow of $G_m$, then $\tilde{t}\rightarrow \tilde{s}$ is an arrow of $G_m$. Since the two cycles of $G_m$ are of different length, we can derive from the fact that only one of the two arrows
\begin{center}
	$R_m\rightarrow R_m^*t_{n-1}^{-}$ \quad and $R_m\rightarrow R_m^*t_{n}^{-}$
\end{center}
belongs to the referent cycle the fact that only one of the two arrows
\begin{center}
	$0L_m^{-}\rightarrow L_m$ and $1L_m^{-}\rightarrow L_m$
\end{center}
belongs to the referent cycle. The two words $0c_\alpha$ and $1c_\alpha$ are Sturmian, and so there is a unique word $u$ of length $q_n$ such that $uc_\alpha$ is Sturmian and turns around the referent cycle. Since $c_\alpha$ always turns at least once around the referent cycle, $uc_\alpha$ turns twice around the referent cycle.

If there is a Sturmian word $x$ that turns twice around the non-referent cycle, wee see from the preceeding argument that the central word $w$ defined by the common part of $G_m$ is such that the four word $0w0$, $1w0$, $0w1$ and $1w1$ are factors of $x$. But this contradicts the balanced property of Sturmian words.
\end{proof}

\section{Formal Intercepts of Sturmian words}

We still consider a slope $\alpha$ with continuants $(q_n)$.

\begin{prop} Let $\displaystyle N=\sum_{i=0}^{k-1}b_{i+1}q_i$ with $b_i\geq 0$ for all $i\geq 2$ and $n\geq 1$. Then the following statements are equivalent :
\begin{enumerate}[i)]
	\item $\forall l=1\ldots k$, $\displaystyle\sum_{i=0}^{l-1}b_{i+1}q_i<q_{l}$
	\item We have :	\begin{itemize}	\item $0\leq b_1 \leq a_1-1$
		\item $\forall i\geq 1$, $0\leq b_i \leq a_i$
		\item $\forall i\geq 1$, $b_{i+1}=a_{i+1}\Rightarrow b_i=0$
	\end{itemize}
\end{enumerate}
\end{prop}

\begin{proof}
\underline{$i)\Rightarrow ii)$ :} Since $q_1=a_1$, the first line of $ii)$ is easily checked. Let $j\geq 1$, then if $b_j> a_j$ we have $b_jq_{j-1}\sum_{i=0}^{j-1}b_{i+1}q_i<q_{j}=a_jq_{j-1}+q_{j-2}\leq (a_j+1)q_{j-1}$ which is absurd. If $b_{j+1}=a_{j+1}$ then from $\sum_{i=0}^{j}b_{i+1}q_i<q_{j+1}=a_{j+1}q_j+q_{j-1}$ we get $\sum_{i=0}^{j-1}b_{i+1}q_i<q_{j-1}$ which clearly implies $b_{j}=0$.

\underline{$ii)\Rightarrow i)$ :} The result is clear for $l=1$ and we prove the result by induction on $l$. Assume $\sum_{i=0}^{l-1}b_{i+1}q_i<q_{l}$. If $b_{l+1}<a_{l+1}$ then $\sum_{i=0}^{l}b_{i+1}q_i<q_{l}+b_{l+1}q_l\leq a_{l+1}q_l < q_{l+1}$. If $b_{l+1}=a_{l+1}$ then by assumption $b_l=0$ so that $\sum_{i=0}^{l-1}b_{i+1}q_i<q_{l-1}$ and $\sum_{i=0}^{l}b_{i+1}q_i<q_{l-1}+a_{l+1}q_l=q_{l+1}$.
\end{proof}

For a sequence $(b_i)_{i\geq 1}$, we call the conditions of Proposition $8$ as the Ostrowski conditions.

\begin{prop} Every integer $N\in [0,q_n[$ can be written uniquely in the form
	\begin{center}
		$\displaystyle N=\sum_{i=0}^{n-1}b_{i+1}q_i$
	\end{center}
	where the integers $(b_i)$ satisfy the Ostrowski conditions.
\end{prop}

\begin{proof}
We proceed by induction on $N$. Write $N=b_nq_{n-1}+c$ with $c\in[0,q_{n-1}[$. By induction, $c$ can be written uniquely in the form $c=\sum_{i=0}^{n-2}b_{i+1}q_i$ where the coefficient $(b_i)_{i=1}^{n-1}$ satisfy the Ostrowski conditions. It is obvious that $b_n\leq a_n$. If $b_n=a_n$, then we must have $c<q_{n-2}$ and by induction on the unicity we must have $b_{n-1}=0$ so that the sequence $(b_i)$ indeed satisfy the Ostrowski conditions.
\end{proof}

\begin{defi}
We define the set :
\begin{center}
	$\displaystyle\mathcal{I}_\alpha=\left\{\left.(k_n)_{n>0}\in\prod_{n>0}[0,q_{n}[ \ \right| \ \forall n\geq 0,\ k_n=k_{n+1} \text{\texttt{\emph{[mod}} }q_n\text{\texttt{\emph{]}}} \right\}$
\end{center}
of formal intercepts of the slope $\alpha$.
\end{defi}

\begin{remark}
In view of proposition 8 and 9, if $\rho=(\rho_n)_{n\geq 0}$ is a formal intercept, there is a unique sequence of integers $(b_i)_{i\geq 1}$, satisfying the Ostrowski conditions, such that
\begin{center}
	$\displaystyle \rho_n=\sum_{i=0}^{n-1}b_{i+1}q_i$
\end{center}
for all $n\geq 0$. In this case, we directly write :
\begin{center}
	$\displaystyle \rho=\sum_{i=0}^{+\infty}b_{i+1}q_i$.
\end{center}
\end{remark}

\begin{remark}
For $n>0$, we define :
\begin{center}
	$\Psi_n^{n+1}$ : $\begin{matrix}
		 [0,q_{n+1}[ & \longmapsto & [0,q_{n}[  \\
		 k & \longmapsto & k $\texttt{ [mod $q_n$]}$  \\
		\end{matrix}$
\end{center}
and for integers $m\geq n >0$ :
\begin{center}
	$\Psi_n^{m}=\Psi_n^{n+1}\circ \Psi_{n+1}^{n+2}\circ \cdots \circ\Psi_{m-1}^{m}\ :\ [0,q_{m}[ \ \rightarrow [0,q_{n}[$
\end{center}
then
\begin{center}
	$\displaystyle \mathcal{I}_\alpha=\lim_{\longleftarrow}[0,q_n[=\left\{\left. (k_n)_{n>0}\in\prod_{n>0}[0,q_{n}[ \ \right| \ n\leq m \Rightarrow \Psi_n^{m}(k_m)=k_n \right\}$
\end{center}
may be viewed as the projective limit of the sets $[0,q_n[$ endowed with the functions $\Psi_n^{m}$.
\end{remark}

\begin{prop}
Let $\rho=\sum_{i\geq 0}b_{i+1}q_i$ be a formal intercept of the slope $\alpha$, and $n\geq 1$. Let	$\lambda_n=q_{n+1}+q_n-\rho_{n+1}-2$, then
\begin{enumerate}
	\item The words $T^{\rho_{n}}(c_\alpha)$ and $T^{\rho_{n+1}}(c_\alpha)$ share the same prefix of length $\lambda_n$.
\item If $b_{n+1}\neq 0$, then $\lambda_n$ is the length of the longest common prefix of $T^{\rho_{n}}(c_\alpha)$ and $T^{\rho_{n+1}}(c_\alpha)$,
\item the increasing sequence $(\lambda_n)_{n\geq 1}$ is unbounded.
\end{enumerate}
\end{prop}
	
\begin{proof}
1) Let $m=q_n-1\in I_n^0$. By theorem $6$, the word $T^{b_{n+1}q_n}(c_\alpha)$ turns $a_{n+1}-b_{n+1}$ times around the referent cycle, and then turns around the non-referent cycle. This shows that the words 
\begin{center}
	$T^{b_{n+1}q_n}(c_\alpha)$ \quad and \quad $c_\alpha$
\end{center}
share the same prefix of length
\begin{center}
	$m+(a_{n+1}-b_{n+1})q_n+r$
\end{center}
where $r$ is the length of the common part of the two cycles of $G_m$. Since $m=q_n-1$, every vertex of the non-referent cycle belongs to the referent cycle. This implies that $r=q_{n-1}-1$, and the two words $T^{b_{n+1}q_n}(c_\alpha)$ and $c_\alpha$ share the same prefix of length $m+(a_{n+1}-b_{n+1})q_n+r=q_n-1+(a_{n+1}-b_{n+1})q_n+q_{n-1}-1$. This shows that the two words
\begin{center}
	$T^{\rho_n}(T^{b_{n+1}q_n}(c_\alpha))=T^{\rho_{n+1}}(c_\alpha)$ \quad and \quad $T^{\rho_n}(c_\alpha)$
\end{center}
share the same prefix of length $q_n+(a_{n+1}-b_{n+1})q_n+q_{n-1}-2-\rho_n=q_{n+1}+q_n-\rho_{n+1}-2=\lambda_n$.

2) If $b_{n+1}\neq0$ then the longest common prefix of the words $T^{b_{n+1}q_n}(c_\alpha)$ and $c_\alpha$ has length $q_n-1+(a_{n+1}-b_{n+1})q_n+q_{n-1}-1$. So that the length of the longest common prefix of $T^{\rho_{n+1}}(c_\alpha)$ and $T^{\rho_n}(c_\alpha)$ indeed equals $\lambda_n$.

3) We have $\lambda_{n+1}-\lambda_n=q_{n+2}+q_{n+1}-q_{n+1}-q_n-(\rho_{n+2}-\rho_{n+1})=(a_{n+2}-b_{n+2})q_{n+1}\geq 0$ so that the sequence $(\lambda_n)$ is increasing. Since $\rho_{n+1}<q_{n+1}$, we get :
\begin{center}
	$\lambda_n\geq q_n-1$
\end{center}
and this shows that $\lambda_n \rightarrow +\infty$ when $n\rightarrow +\infty$.
\end{proof}
	
\begin{remark}
Notice that in the case where $\rho_{n+1}=q_{n+1}-1$ then $\lambda_n= q_n-1$ and the lower bound for $(\lambda_n)$ found in the proof of $3)$ is optimal. However, the sequence $(q_n-1)_{n\geq 1}$ does not defines a formal intercept.
\end{remark}
	
\begin{defi}
Let $\rho$ be a formal intercept of the slope $\alpha$. We define the Sturmian word $T^{\rho}(c_\alpha)$ of slope $\alpha$ and formal intercept $\rho$ as the word
\begin{center}
	$T^{\rho}(c_\alpha)=\lim T^{\rho_n}(c_\alpha)$
\end{center}
having the same prefix of length $q_n-1$ as $T^{\rho_n}(c_\alpha)$ for all $n\geq 1$.
\end{defi}
	
\begin{prop}
Let $\rho$ be a formal intercept of the slope $\alpha$ ans $n\geq 1$. Then the length of the longest common prefix of the words
\begin{center}
	$T^{\rho}(c_\alpha)$ \quad and \quad $T^{\rho_n}(c_\alpha)$
\end{center}
equals $\lambda_N$, where $N$ is the smallest integer $N\geq n$ such that $b_{N+1}\neq 0$. If no such $N$ exists, then they are equal.
\end{prop}

\begin{proof}
This is clear, since $\rho_n=\rho_k$ for all $n\leq k \leq N$ if such a $N$ exists, and $\rho_n=\rho_k$ for all $n\leq k$ in the second case.
\end{proof}
	
\begin{prop}
Let $x$ be a Sturmian word of slope $\alpha$. Then there exist a unique formal intercept $\rho$ of the slope $\alpha$ such that $x=T^{\rho}(c_\alpha)$.
\end{prop}

\begin{proof}
We consider the sequence, defined for $n\geq 0$ as :
\begin{center}
	$\rho_n=\min\{k\geq 0 \ | \ x \text{ and } T^k(c_\alpha) \text{ share the same prefix of length } q_n-1\}$
\end{center}
and show that $\rho=(\rho_n)_{n\geq 1}$ is a formal intercept. Let $n\geq 1$ and $m=q_n-1\in I_n^0$. Since the referent cycle is of length $q_n=m+1$, every vertex of $G_m$ is in the referent cycle. This shows that $0\leq\rho_n<q_n$. Since $T^{\rho_{n+1}}(c_\alpha)$, $T^{\rho_{n}}(c_\alpha)$ and $x$ share the same prefix of length $q_n-1$, the paths they define start at the same vertex.

Write $\rho_{n+1}=bq_n+c$ with $c<q_n$. Since $\rho_{n+1}=bq_n+c<q_{n+1}=a_{n+1}q_n+q_{n-1}$, we have $b\leq a_{n+1}$ and if $b= a_{n+1}$ then $c< q_{n-1}$. Since the caracteristic word $c_\alpha$ turns $a_{n+1}$ times around the referent cycle, if $b<a_{n+1}$ then $T^{\rho_{n+1}}(c_\alpha)$ and $T^{c}(c_\alpha)$ start at the same vertex, and hence share the same prefix of length $q_n-1$. Since the referent cycle is of length $q_n$, and that $\rho_n<q_n$ we must have $c=\rho_n$. In the case where $b=a_{n+1}$, then $c<q_{n-1}$ so that $T^{\rho_{n+1}}(c_\alpha)$ starts in the common part of the two cycles of $G_m$, $T^{\rho_{n+1}}(c_\alpha)$ and $T^{c}(c_\alpha)$ start at the same vertex, which is on the referent cycle, and we again must have $\rho_n=c$. Thus $\rho_n=\rho_{n+1} $\texttt{ [mod $q_n$]} and we are done.

For unicity, notice that since for $m=q_n-1$ the referent cycle is of length $q_n$, there must be only one $k<q_n$ such that $T^{k}(c_\alpha)$ and $T^{\rho}(c_\alpha)$ share the same prefix of length $q_n-1$, and since $\rho_n$ is such a $k$, every formal intercept $\gamma$ such that $x=T^{\gamma}(c_\alpha)$ must satisfy $\gamma_n=\rho_n$.
\end{proof}

Example : One can compute easily that the infinite words $0c_\alpha$ and $1c_{\alpha}$ have respective formal intercepts $\sum_{i\geq 0} a_{2i+2}q_{2i+1}$ and $(a_1-1)+\sum_{i\geq 1} a_{2i+1}q_{2i}$.

\underline{\textit{Remark :}}
In a future paper we will investigate more properties of formal intercepts.


\begin{thebibliography}{99}

\bibitem{lothaire}
\textsc{M. Lothaire}, \textit{Algebraic Combinatorics on words}, (2000) ISBN :  9781107326019

\bibitem{shallit}
\textsc{J.-P. Allouche, J. O. Shallit}, \textit{Automatic Sequences : Theory, Applications, Generalizations}, (2002) ISBN :  9780521823326

\bibitem{knuth}
\textsc{D. E. Knuth}, \textit{Fibonacci multiplication}, Appl. Math. Lett., 1 (1988), pp. 57-60

\bibitem{arnoux}
\textsc{P. Arnoux}, \textit{Some remarks about Fibonacci multiplication}, Appl. Math. Lett., 2 (1989), pp. 319-320

\bibitem{berthe}
\textsc{V. Berth\'e}, \textit{Autour du syst\`eme d'\'enum\'eration d'Ostrowski}, Bull. Belg. Math. Soc. 8 (2001), pp. 209-238

\bibitem{berthe}
\textsc{V. Berth\'e}, \textit{Fr\'equences des facteurs des suites sturmiennes}, 
Theoretical Computer Science, Volume 165, Issue 2, 1996,
pp. 295-309, ISSN 0304-3975,
https://doi.org/10.1016/0304-3975(95)00224-3.

\bibitem{holton}
\textsc{V. Berth\'e, C Holton, Luca Q. Zamboni}, \textit{Initial powers of Sturmian sequences}, 
Acta Arithmetica, Instytut Matematyczny PAN, 2006, 122, pp.315-347. lirmm-00123046

\bibitem{bugeaud}
\textsc{Y. Bugeaud, D. H. Kim}, \textit{A new complexity function, repetitions in sturmian words}, 	arXiv:1510.00279 [math.NT]

\bibitem{mayero}
\textsc{M. Mayero}, \textit{The Three Gap Theorem},
arXiv:cs/0609124 [cs.LO]

\bibitem{siegel}
\textsc{A. Siegel}, \textit{Th\'eor\'eme des trois longueurs et suites sturmiennes : mots d'agencement des longueurs}, ACTA ARITHMETICA
XCVII.3 (2001)

\bibitem{ramshaw}
\textsc{L. Ramshaw}, \textit{On the discrepancy of the sequence formed by the multiples of an irrational number}, Journal of Number Theory, Volume 13, Issue 2, 1981, pp. 138-175, ISSN 0022-314X,
https://doi.org/10.1016/0022-314X(81)90002-0.



\end{thebibliography}
\end{document}